\documentclass[10pt]{article}

\usepackage{amsmath}
\usepackage{amssymb}
\usepackage{amsthm, url}

\newtheorem{theorem}{Theorem}
\newtheorem{corollary}[theorem]{Corollary}
\newtheorem{lemma}[theorem]{Lemma}

\newtheorem{remark}[theorem]{Remark}

\def\Z{\mathbb{Z}}
\def\Q{\mathbb{Q}}

\def\OK{\mathcal{O}_K}
\def\eps{\varepsilon}

\DeclareMathOperator{\Disc}{Disc}

\DeclareMathOperator{\Aut}{Aut}

\title{Counting quintic fields with genus number one}
\author{Kevin J.~McGown \and Frank Thorne \and Amanda Tucker}

\begin{document}

\maketitle

\section{Introduction}

The genus field of a number field $K$ is defined to be the maximal extension $K^*$ of $K$
that is unramified at all finite primes and is a compositum of the form $Kk^*$
where $k^*$ is absolutely abelian.  The genus number is defined as $g_K=[K^*:K]$.
It follows immediately from class field theory that $g_K$ divides the narrow class number $h_K^+$.
See~\cite{mcgown.tucker} for a slightly more detailed introduction
and~\cite{ishida} for a comprehensive account of genus fields.

One may consider the density
of genus number one fields among all number fields of a fixed degree and signature,
ordered by their discriminants.  It essentially follows from a classical theorem of Gauss
on quadratic forms
that $0\%$ of quadratic fields have genus number one.\footnote{Indeed, Gauss proves (in the language of quadratic forms) that when $K$ is quadratic, $g_K=2^{t-1}$ where $t$ is the number of prime divisors of $\Disc(K)$, and moreover, that $g_K$ in fact equals the $2$-part of $h_K^+$ in that situation.  It was Hasse who first reproved Gauss' result in the language of class field theory, by considering the genus field in the quadratic setting.}
On the other hand,
McGown and Tucker (\cite{mcgown.tucker}) proved that a positive proportion (roughly $96.23\%$)
of cubic fields have genus number one.  Due to subtleties that arise in the quartic case,
we will temporarily put degree four fields aside and consider this 
problem in the quintic case.

Let $\mathcal{F}$ denote the collection of all quintic fields $K$,
and $\mathcal{G}$ denote the collection of all quintic fields with $g_K=1$.
For $i=0,1,2$,
write $\mathcal{F}^{(i)}$ and $\mathcal{G}^{(i)}$ to denote the subsets of $\mathcal{F}$
and $\mathcal{G}$, respectively, consisting of fields with precisely~$i$ pairs of complex embeddings.
Define $N^{(i)}(X)=\#\{K\in\mathcal{F}^{(i)}:|\Disc(K)|\leq X\}$ and
$N_g^{(i)}(X)=\#\{K\in\mathcal{G}^{(i)}:|\Disc(K)|\leq X\}$.
Bhargava proved (see Theorem~1 of~\cite{bhargava.quintic}, pg.~1559) that
$N^{(i)}(X)\sim C^{(i)}X$
where
\[
  C^{(i)}=\prod_p(1+p^{-2}-p^{-4}-p^{-5})\cdot
  \begin{cases}
  \frac{1}{240} & \text{if $r=0$}\\  \frac{1}{24} & \text{if $r=1$} \\  \frac{1}{16} & \text{if $r=2$}\,.
  \end{cases}
\]

Our main result concerning counting genus number one fields is the following:
\begin{theorem}\label{T:1}
\[
  N_g^{(i)}(X)=\left(C^{(i)}\frac{506874}{506875}\prod_{p\equiv 1\pmod{5}} \frac{p^4+p^3+2p^2+2p}{p^4+p^3+2p^2+2p+1} \right)X + O(X^{1-\frac{1}{400} + \eps})
  \,.
\]
\end{theorem}
\begin{corollary}
  The proportion of quintic fields with $i$ pairs of complex embeddings having genus number one equals
  \[
  \frac{506874}{506875}
  \prod_{p\equiv 1\pmod{5}} \frac{p^4+p^3+2p^2+2p}{p^4+p^3+2p^2+2p+1}
  \,.
  \]
\end{corollary}
Although the number above (which is independent of $i$ and approximately equal to $0.999935$) is quite close to $1$, this implies
that a positive proportion of quintic fields have class number divisible by $5$.
(Note that since $g_K$ divides the narrow class number and $g_K$ is a power of $5$,
we also have that $g_K$ divides the class number.)
In fact, our methods yield the following stronger result:
\begin{theorem}\label{T:2}
Given any $k\geq 0$, there is a positive proportion of quintic fields with
genus number equal to $5^k$, and hence with class number divisible by $5^k$.
\end{theorem}
We also prove the following:
\begin{theorem}\label{T:3}
The average of the genus number taken over all quintic fields of a given signature is finite.
In particular, we have
$$
\lim_{X\to\infty}
\frac{\sum_{\substack{K\in\mathcal{G}^{(i)} \\|\Disc(K)|\leq X}} g_K}{\sum_{\substack{K\in\mathcal{G}^{(i)} \\|\Disc(K)|\leq X}} 1}
=
\frac{506879}{506875}
\prod_{ p \equiv 1 \pmod{5}}
\frac{p^4 + p^3 + 2p^2 + 2p + 5}{p^4 + p^3 + 2p^2 + 2p + 1}
\,.
$$
\end{theorem}
Numerically, this constant is $1.00026\dots$, but the main point is that the limit exists.

Late in the preparation of this paper, we became aware of the work of Kim (see~\cite{kim}) on
 statistical questions concerning the genus number in cyclic and dihedral extensions of prime degree, in which he says ``it seems very difficult to compute genus numbers of [$S_5$] fields''.
All of the results in this paper hold, with essentially identical proofs, when restricted to $S_5$-quintic fields only. 
As such, our results establish that a variety of statistics can indeed be computed for genus numbers of such fields.

Finally, we also show the following:
\begin{theorem}\label{T:4}
A positive proportion of quintic fields with genus number one fail to be norm-Euclidean.
\end{theorem}

\section{Counting quintic fields with specified local completions}

Our proofs will apply two results on counting quintic fields. The first, essentially due to Ellenberg, Pierce, and Wood 
\cite[Theorem 5.1]{ellenberg.pierce.wood}, and building on Bhargava's work
\cite{bhargava.quintic},
is a result counting quintic fields
with a finite specified set of local conditions. 

By a {\itshape local condition} $\Sigma_p$ at a prime $p$, we mean
some fixed subset $\Sigma_p$ of the quintic \'etale algebras over $\Q_p$,
and we say that a quintic (number) field $K$ satisfies $\Sigma_p$ if $K\otimes \Q_p\in\Sigma_p$.
For squarefree $e$, 
by a set of local conditions $\Sigma \pmod{e}=(\Sigma_p)_{p \mid e}$, we mean a choice of local condition
$\Sigma_p$ for each prime $p \mid e$, and we say that $K$ satisfies $\Sigma$ if it satisfies each of the $\Sigma_p$.

To avoid a technical complication, we assume, for each $p \mid e$ other than $p = 5$, that $\Sigma_p$ corresponds to a
`splitting type' -- i.e., that $\Sigma_p$ consists of those algebras $K_{v_1} \times \dots \times K_{v_g}$ for which $g$, each
$e(K_{v_i}|\Q_p)$, and each $f(K_{v_i}|\Q_p)$ take a prescribed value. We consider such an algebra to be ramified if any of the
$K_{v_i}$ is ramified, and we 
write $e = e_1 e_2$ or $e = 5e_1 e_2$, where $e_1$ and $e_2$ are the products of those
primes $p \neq 5$ for which $\Sigma_p$ consists of unramified or ramified algebras, respectively.

Define the function
$$
  m(p)
  =
  1+p^{-1}+2p^{-2}+2p^{-3}+p^{-4}
  \,,
$$
which is the total mass of all quintic \'etale extensions of $\Q_p$.
Indeed, Bhargava
(see Theorem~1.1 of~\cite{bhargava.mass})
proves a generalization of Serre's mass formula (Th\'eor\`eme 2 of~\cite{serre})
which, specialized to the quintic case, yields
\[
  \sum_{[K:\Q_p]=5\text{ \'etale}}
  \frac{1}{\Disc_p(K)}\frac{1}{\#\Aut(K)}=m(p)
  \,.
\]
For the condition that $p$ not totally ramify, i.e.,
$$\Sigma_p=\{[F:\Q_p]=5 \text{ \'etale}\mid e(F/\Q_p)<5\}\,,$$
we have
\[
\sum_{F\in\Sigma_p}\frac{1}{\Disc_p(F)}\frac{1}{\#\Aut(F)} = 1+p^{-1}+2p^{-2}+2p^{-3}
\,,
\]
whereas for the condition that $p$ totally ramify the same sum is equal to $p^{-4}$, as originally proved by Serre. 

The following theorem is adapted from a result of
Ellenberg--Pierce--Wood (see Theorem~2.4 of~\cite{ellenberg.pierce.wood}).

\begin{theorem}
\label{T:counting}
Let $N^{(i)}(X; \Sigma)$ be the number of quintic fields $K$ with $|\Disc(K)|<X$
having $i$ pairs of complex embeddings, and
satisfying a set of local conditions $\Sigma \pmod{e}$ with the restriction described above. Then, we have
\begin{align*}
  N^{(i)}(X;\Sigma)=C^{(i)}(\Sigma)
  X+O\left(e_1^{1/2} e_2 X^{79/80 + \eps} + X^{199/200 + \eps} \right),
\end{align*}
where 
\begin{align*}
C^{(i)}(\Sigma)&=C^{(i)}\prod_p C_p(\Sigma_p), \\
  C_p(\Sigma_p)
  &=
  m(p)^{-1}\sum_{F\in\Sigma_p}\frac{1}{\Disc_p(F)}\frac{1}{\#\Aut(F)}
  \,,
  \\
  m(p)
  &=
  1+p^{-1}+2p^{-2}+2p^{-3}+p^{-4}
  \,.
\end{align*}
\end{theorem}

\begin{proof} The proof is essentially identical to that of \cite[Theorem 5.1]{ellenberg.pierce.wood}.
The explanation given there is quite thorough, and we will indicate only what needs to be changed.

As in \cite{ellenberg.pierce.wood}, the proof applies an inclusion-exclusion argument to points in the lattice
$V_{\Z} = \Z^{40}$ studied in \cite{bhargava.quintic}. For each squarefree integer $q$ coprime to $5e$, let
$W_{q, e} \subset V_\Z$ denote the set of elements corresponding to quintic rings that are:
nonmaximal at each prime dividing $q$; maximal and which satisfy the desired local conditions at primes dividing $5e$. 

Write $e_1$ for the product of primes $p \mid e$, excluding $5$, for which $\Sigma_p$ consists of unramified algebras,
and write $e_2$ for the analogous product where
$\Sigma_p$ consists of ramified algebras.

Then, $W_{q, e}$ is defined by congruence conditions modulo $q^2 e_1 e_2^2 5^k$, for some positive integer $k$.
This contrasts to the situation in \cite{ellenberg.pierce.wood},
where no special conditions modulo $5$ were imposed, and where the conditions modulo $e$ were all unramified. 
In the ramified case, the splitting type $(1^5)$ is defined $\pmod{e}$, as described in 
\cite[Section 12]{bhargava.quintic}, but the maximality condition is no longer automatic and this is defined only $\pmod{e^2}$.

The remainder of the analysis remains essentially unchanged. The error term of \cite[(5.3)]{ellenberg.pierce.wood} includes an error term $q^2 \delta_P e$,
and $e$ must be replaced by $e_1 e_2^2$ as described above. (The $5^k$ term contributes to the implied constant, and we may
ignore this contribution.) We now have $\delta_P \ll e_2^{-1}$, reflecting the fact that the fields being counted are rare. 
Finally, in invoking \cite[(27)]{bhargava.quintic}, it must be assumed that $q^2 e_1 e_2^2 \ll X^{1/40}$.

The ensuing analysis then remains valid, with identical or improved bounds on the error terms $E_1, E_2, E_3, E_4$. The restriction
that $q^2 e_1 e_2^2 \ll X^{1/40}$ proves to be the bottleneck, and choosing $Q = X^{1/80} e_1^{-1/2} e_2^{-1}$ completes the proof.
\end{proof}

\begin{remark} For a general set of local conditions $\Sigma \pmod{e}$, analogous results hold with an undetermined 
$e$-dependence in the error term. As the proof shows, this dependence can be computed in terms 
of a modulus for which $W_{q, e}$ can be defined by congruence
conditions. 
\end{remark}

We will also apply the following complementary `tail estimate', which will be contained in the forthcoming article
\cite{bhargava.cojocaru.thorne} by Bhargava, Cojocaru, and the second author:
\begin{theorem}\label{thm:tail}
For any $Y > 1$, define
\begin{equation}
N_5(X, Y) := 
\# \left\{ K \textnormal{ quintic} : \ |\Disc(K)| \leq X, \ q^2 \mid \Disc(K) \textnormal{ for some squarefree } q > Y \right\}.
\end{equation}
Then 
\[
N_5(X, Y) \ll_\varepsilon X^{\frac{39}{40} + \varepsilon} + \frac{X}{Y^{1 - \varepsilon}}.
\] 
\end{theorem}

\section{The genus theory of quintic fields}

Let $K$ be a non-cyclic quintic field.
For $p\neq 5$, write $k(p)$ to denote the unique quintic subfield of $\Q(\zeta_p)$ if such a field exists and $\Q$ otherwise.
Ishida proves that
$Kk(p)/K$ is a nontrivial unramified extension if and only if $p$ is totally ramified in $K$ and
$p\equiv 1\pmod{5}$.

We now consider the case of $p=5$.
Write $k(5)$ to denote the unique quintic subfield of $\Q(\zeta_{25})$.
Ishida shows that $Kk(5)/K$ is a nontrivial unramified extension if and only if $5$ is totally ramified and
$N(\gamma)^{4}\equiv 1\pmod{25}$ for all $\gamma\in\OK$ coprime to $5$.
We will refer to this latter congruence condition as condition~$(\star)$.
Ultimately, Ishida proves the following result
(see Equation~5.9 of~\cite[Chapter 5]{ishida}, pg.~65).

\begin{theorem}[Ishida]\label{T:ishida}
Let $L$ be a quintic field.
Let $t$ denote the number of primes $p$ such that $p$ is totally ramified in $L$ and $p\equiv 1\pmod{5}$,
and add $+1$ to $t$ if $5$ is totally ramified in $L$ and
$N(\gamma)^{4}\equiv 1\pmod{25}$ for all $\gamma\in\mathcal{O}_L$ coprime to $5$.
Then we have:
\[
   g_K=\begin{cases}
   5^{t-1} & \text{if $L$ is cyclic}\\
   5^t & \text{if $L$ is not cyclic}
   \end{cases}
\]
\end{theorem}

We wish to reformulate condition $(\star)$ in a manner suitable for our calculations.
Suppose $5$ is totally ramified in $K$.
Let $\mathfrak{p}$ be the unique prime in $K$ above $5$.
It is plain that the congruence condition above involving the norm can be checked in the quintic extension
$K_\mathfrak{p}/\Q_5$ of local fields.  It is equivalent to requiring that
$N(u)^4\equiv 1\pmod{25}$ for all $u\in\mathcal{O}_{\mathfrak{p}}^\times$.
Since every element of $\mathcal{O}_{\mathfrak{p}}^\times$ is a $4$-th root of unity
times an element of the principal units $\mathcal{U}^{(1)}=1+\mathfrak{p}$, it suffices
to check the congruence for all $u\in\mathcal{U}^{(1)}$.

Since $5$ is totally ramified in $K_\mathfrak{p}$ it is possible to choose a generating polynomial
$f(x)=x^5 + a_4x^4 + a_3x^3 +a_2x^2 +a_1x+a_0$ that is Eisenstein at the prime $5$.
A simple calculation (given on~pages 57--59 of~\cite{ishida})
shows that the condition
$N(u)^4\equiv 1\pmod{25}$ for all $u\in\mathcal{U}^{(1)}$
is equivalent to
\[
  a_1 \equiv a_2 \equiv a_3 \equiv a_4+a_0 \equiv 0 \pmod{5^2}
  \,.
\]
(In truth, Ishida gives one direction of this claim, but the argument is easily reversible.)
There are $25$ quintic ramified extensions of $\Q_5$.  See~\cite{jones.roberts, amano}.  For each of these, we use the generating polynomial given in the Jones--Roberts database
(now incorporated into the LMFDB \cite{lmfdb})
 to check condition ($\star)$.
We find that precisely $5$ of these $25$ extensions satisfy the condition.
Generating polynomials for these extensions are given by
$x^5-5x^4+5(1+5a)$ for $0\leq a \leq 4$.
Moreover, these are precisely the Galois extensions.
We mention in passing that all of these extensions have discriminant $5^8$.
We have thus proved the following two results.

\begin{lemma}\label{5totram}
Suppose $K$ is a non-cyclic quintic field where $5$ is totally ramified.
Let $k(5)$ be the unique quintic subfield of $\Q(\zeta_{25})$.
We have that $Kk(5)/K$ is unramified if and only if
$K\otimes \Q_5$
is a Galois extension of $\Q_5$.
\end{lemma}

\begin{lemma}\label{5mass}
The mass of all totally ramified quintic Galois extensions of $\Q_5$ is equal to
\[
  \sum_{\substack{[K:\Q_5]=5 \text{ Galois,}\\\text{ totally ramified}}}  
    \frac{1}{\Disc_5(K)}\frac{1}{\#\Aut(K)}
  =\frac{1}{5^8}
  \,.
\]
\end{lemma}


\section{Proofs of Theorems~\ref{T:1},~\ref{T:2},~\ref{T:3},~\ref{T:4}.}

Recall that $\mathcal{F}$ denotes the collection of all quintic fields $K$,
and that $\mathcal{G}$ denotes the collection of all quintic fields with $g_K=1$.

\begin{proof}[Proof of Theorem~\ref{T:1}]
Theorem \ref{T:ishida} and Lemma \ref{5totram}
establish that $\mathcal{G}$, the set of non-cyclic quintic fields $K$ 
with genus number $g_K = 1$, consists precisely of those $K$ satisfying the following local conditions:
\begin{itemize}
\item No prime $p \equiv 1 \pmod{5}$ is totally ramified in $K$.
\item Either $5$ is not totally ramified in $K$, or $K\otimes \Q_5$ is not a Galois field extension of $\Q_5$.
\end{itemize}

We count these fields (and hence prove Theorem~\ref{T:1}) by means of an inclusion-exclusion sieve, adapting the approach of
Belabas, Bhargava, and Pomerance in \cite{belabas.bhargava.pomerance}. Throughout this section, we ignore the cyclic 
quintic fields; by Theorem 1.1 of \cite{CoDiOl3} there are $\sim c X^{1/4}$ of them with discriminant bounded by $X$, and hence they
do not contribute to any of our asymptotics. Alternatively, one may exclude all of the non-$S_5$ quintic fields, of which there
are $\ll X^{39/40 + \epsilon}$ by 
\cite{bhargava.cojocaru.thorne}.

For each quintic field $K$, let $f(K)$ denote the product of
the primes $p \equiv 0, 1 \pmod{5}$ that are totally ramified in $K$, with the
additional condition for $p = 5$ that $K\otimes \Q_5$ be a Galois field extension of $\Q_5$.
Then $\mathcal{G}^{(i)}$ is the set of all $K$ in $\mathcal{F}^{(i)}$
for which $f(K) = 1$.
Recall that $N_{g}^{(i)}(X)$ denotes the counting function for $\mathcal{G}^{(i)}$.

Let $T$ denote the collection of all positive squarefree $f$ whose prime divisors $p$ all satisfy
$p \equiv 0, 1\pmod{5}$. For each $f \in T$, write $\Sigma_f$
for the set of local conditions specifying that
$f \mid f(K)$.
Let $\mathcal{A}^{(i)}_f$ be the set of fields in $\mathcal{G}^{(i)}$ satisfying the conditions $\Sigma_f$, and
observe that
$$
  \mathcal{G}^{(i)}=\mathcal{F}^{(i)}\setminus
  \bigcup_{p\equiv 0,1\pmod{5}} \mathcal{A}^{(i)}_p
  \,.
$$
Applying inclusion-exclusion
it follows, for an arbitrary parameter $Y < X$, that 
\begin{align}
N_{g}^{(i)}(X) & = \sum_{f\in T} \mu(f) N^{(i)}(X;\Sigma_f) \label{eq:ng1} \\
& = \sum_{\substack{f\in T \\ f \leq Y}} \mu(f) N^{(i)}(X;\Sigma_f) + \nonumber
 \sum_{\substack{f\in T \\ f > Y}} \mu(f) N^{(i)}(X;\Sigma_f).
 \end{align}
As $f^4 \mid \Disc(K)$ for each $K$ satisfying $\Sigma_f$, the sum over $f > Y$ is handled by the tail estimate of Theorem \ref{thm:tail}. Each field $K$ 
is counted with multiplicity at most
$d(\Disc(K)) = O(X^{\eps})$ in this sum (where $d(n)$ is the number of positive divisors of $n$), so that
we have
\begin{equation}\label{eq:tail}
 \sum_{\substack{f\in T \\ f > Y}} \mu(f) N^{(i)}(X;\Sigma_f)
 \ll X^{39/40 + \eps} + X^{1 + \eps}/Y.
 \end{equation}

Define $m^*(p)$ to be $m(p)$ for $p \neq 5$, with $m^*(5) = 5^4 m(5)$.
For each $f \leq Y$, upon 
applying Theorem~\ref{T:counting} with Lemma~\ref{5mass} we obtain
\begin{align*}
N^{(i)}(X;\Sigma_f)&=\left( C^{(i)} \prod_{p|f}m^*(p)^{-1}p^{-4} \right) \cdot X + O(X^{199/200 + \eps})
\,,
\end{align*}
provided that $Y \leq X^{\frac{199}{200} - \frac{79}{80}} = X^{\frac{3}{400}}$,
so that 
\begin{equation}\label{eq:avg1}
\sum_{\substack{f\in T \\ f \leq Y}} \mu(f) N^{(i)}(X;\Sigma_f)
=
C^{(i)}
X
\sum_{\substack{f\in T \\ f \leq Y}} \left( \prod_{p|f}-m^*(p)^{-1}p^{-4} \right) + O\left(Y X^{199/200 + \eps}\right).
\end{equation}
In the main term we extend the sum over $f \leq Y$ to all $f \in T$, at the expense of an error term
$\ll X Y^{-3}$, since each product over $p \mid f$ has absolute value $\leq f^{-4}$. We thus have that
\begin{align*}
\sum_{\substack{f\in T \\ f \leq Y}} \mu(f) N^{(i)}(X;\Sigma_f)
& =
C^{(i)}
 X
\sum_{\substack{f\in T}} \left( \prod_{p|f}-m^*(p)^{-1}p^{-4} \right) + O\left( Y X^{199/200 + \eps} + X Y^{-3} \right).
\end{align*}

Choosing $Y = X^{\frac{1}{400}}$, we see that the error terms above and in \eqref{eq:tail} are all $O(X^{1 -\frac{1}{400} + \eps})$.
We therefore have that
\begin{align*}
N_{g}^{(i)}(X)& =
C^{(i)}
X
\sum_{\substack{f\in T}} \left( \prod_{p|f}-m^*(p)^{-1}p^{-4} \right) + O\left( X^{1 - \frac{1}{400} + \eps} \right)
\\
& =
C^{(i)} X \cdot
\left(1 - \frac{1}{5^{8}m(5)} \right)
\prod_{ p \equiv 1 \pmod{5}} \left( 1-m(p)^{-1}p^{-4} \right) + O\left( X^{1 - \frac{1}{400} + \eps} \right),
\end{align*}
which is what we wanted to prove.
\end{proof}

\begin{proof}[Proof of Theorem~\ref{T:2}]

For a parameter $Z$, let $U$ denote any set of $k$ primes $p\equiv 0, 1\pmod{5}$ with $p \leq Z$,
and write $T(U)$ for the set of squarefree $f$ whose prime factors $p$ all satisfy 
$p \equiv 0, 1 \pmod{5}$ and $p \not \in U$.

We modify the proof of Theorem \ref{T:1} by replacing $T$ with $T(U)$, and by adding the 
condition throughout that each $p \in U$ totally ramify and that $K \otimes \Q_5$ be Galois
if $5 \in U$. All such fields have genus number $g_K = 5^k$, and by an identical argument the
number of such fields $K$ with $|\Disc(K)|\leq X$ is
\begin{equation*}
C^{(i)}
\left( \prod_{\substack{ p \in T(U)}} \left( 1-m^*(p)^{-1}p^{-4} \right) \right) \cdot
\left( \prod_{p \in U} m^*(p)^{-1} p^{-4} \right) X 
+ O_{Z}\left( X^{1 - \frac{1}{400} + \eps} \right).
\end{equation*}
Adding over all choices of $U$ with each $p \leq Z$, we obtain an expression of the form $C_k^{(i)}(Z) X + O_{Z}\big( X^{1 - \frac{1}{400} + \eps} \big)$,
for a sequence $C_k^{(i)}(Z)$ which increases with $Z$ and is bounded above by $C^{(i)}$, and which therefore converges to a fixed constant  $C^{(i)}_k$.

This counts all quintic fields with $|\Disc(K)| < X$, with the specified signature, and with $g_K = 5^k$ -- with the exception of those for which any prime $p > Z$
is totally ramified. By Theorem~\ref{thm:tail}, the number of such is $\ll X^{\frac{39}{40} + \varepsilon} + X/Z^{1 - \epsilon}$. Therefore, letting $Z \rightarrow \infty$,
we see that the total number of fields being counted is $C^{(i)}_k X + o(X)$.
\end{proof}


\begin{proof}[Proof of Theorem~\ref{T:3}]
Let $\omega(n)$ denote the number of prime divisors of $n$,
and again let $T$ denote the collection of all positive squarefree $f$
whose prime divisors $p$ all satisfy $p \equiv 0, 1\pmod{5}$.

By construction we have $g_K = 5^{\omega(f(K))}$, with $f(K)$ defined as before, and hence also
\[
g_K = \sum_{f \mid f(K)} 4^{\omega(f)},
\]
so that in analogy with \eqref{eq:ng1} we have
\[
\sum_{\substack{K\in\mathcal{G}^{(i)} \\|\Disc(K)|\leq X}} g_K = 
\sum_{f \in T} 4^{\omega(f)} N^{(i)}(X; \Sigma_f).
\]
This sum is evaluated exactly as in the proof of Theorem \ref{T:1}, with $- m^*(p)$ replaced by $4 m^*(p)$ at every occurrence. 
All of the error terms satisfy identical bounds up to a factor of $O(X^{\epsilon})$ since $4^{\omega(f)} \ll X^{\epsilon}$
for $f \leq X$. We therefore conclude that 
\begin{equation}\label{eq:average_genus}
\sum_{\substack{K\in\mathcal{G}^{(i)} \\|\Disc(K)|\leq X}} g_K = 
C^{(i)}
\left(1 + \frac{4}{5^{8}m(5)} \right) X
\prod_{ p \equiv 1 \pmod{5}} \left( 1+4m(p)^{-1}p^{-4} \right) + O\left( X^{1 - \frac{1}{400} + \eps} \right),
\end{equation}
as desired.
\end{proof}

\begin{proof}[Proof of Theorem~\ref{T:4}]
Let $\Sigma$ denote the local conditions that
$2$ is inert, $5$ is inert, $7$ is totally ramified,
and no prime $p\equiv 1\pmod{5}$ is totally ramified.
Observe that any quintic field $K$ satisfying the conditions $\Sigma$ must have genus
number one.
By way of contradiction, suppose $K$ is norm-Euclidean and satisfies the conditions $\Sigma$.
Let $\mathfrak{p}$ denote the unique prime lying over $7$.
Then there exists $\alpha\in\OK$
such that
$4\equiv \alpha\pmod{\mathfrak{p}}$
with
 $|N(\alpha)|<|N(\mathfrak{p})|=7$.
It follows that $2\equiv 4^5\equiv N(\alpha)\pmod{7}$ and therefore
$N(\alpha)\in\{2,-5\}$.  We are forced to conclude that either
$2$ or $5$ is not inert, a contradiction.
The result now follows from techniques similar to the proof of Theorem~\ref{T:1}.
\end{proof}

\section*{Acknowledgments}

We would like to thank Arul Shankar, Melanie Matchett Wood, and several anonymous referees for helpful comments.

This work was supported by a grant from the Simons Foundation (No. 586594, F.T.)
and an internal research grant through California State University, Chico.

\bibliographystyle{alpha}

\vspace{1ex}
\footnotesize{
\noindent
Kevin J. McGown\\
Department of Mathematics and Statistics\\
California State University, Chico\\
Chico, CA 95929\\
\emph{E-mail address:} {\tt kmcgown@csuchico.edu}\\[1ex]

\noindent
Frank Thorne\\
Department of Mathematics\\
University of South Carolina\\
Columbia, SC 29208\\
\emph{E-mail address:} {\tt thorne@math.sc.edu}\\[1ex]

\noindent
Amanda Tucker\\
Department of Mathematics\\
University of Rochester\\
Rochester, NY 14627\\
\emph{E-mail address:} {\tt amanda.tucker@rochester.edu}
}

\end{document}